\documentclass[twoside, 11pt]{amsart}

\usepackage[T1]{fontenc}
\usepackage[utf8]{inputenc}
\usepackage[english]{babel}
\usepackage{indentfirst}
\usepackage{amsmath, amssymb, amscd, amsthm, amsfonts}
\usepackage{amssymb, amsfonts}
\usepackage{amsthm, amscd}

\usepackage{inputenc}
\usepackage{mathtools}
\usepackage{hyperref}
\usepackage{verbatim}
\usepackage{color}
\usepackage{enumerate}
\usepackage{mathrsfs}

\usepackage{tikz}
\usepackage{pgfplots}
\pgfplotsset{compat=1.11}
\usetikzlibrary{calc}
\usetikzlibrary{decorations.markings}
\usetikzlibrary{arrows}
\usetikzlibrary{hobby}
\usetikzlibrary{quotes}
\usetikzlibrary{angles}
\usepackage{pgfplots}
\usepgfplotslibrary{fillbetween}

\theoremstyle{definition}
\newtheorem{definition}{Definition}

\theoremstyle{plain}
\newtheorem{theorem}[definition]{Theorem}

\theoremstyle{plain}

\theoremstyle{plain}

\theoremstyle{plain}

\theoremstyle{definition}
\newtheorem{remark}[definition]{Remark}
\theoremstyle{plain}
\newtheorem{lemma}[definition]{Lemma}
\theoremstyle{plain}
\newtheorem{proposition}[definition]{Proposition}
\theoremstyle{plain}
\newtheorem{corollary}[definition]{Corollary}
\theoremstyle{definition}

\theoremstyle{definition}

\DeclareMathOperator{\di}{div}

\DeclareMathOperator{\diam}{diam}

\DeclareMathOperator{\graph}{graph}
\DeclareMathOperator{\Ric}{Ric}

\newcommand{\R}{\mathbb{R}}
\newcommand{\Reals}{\mathbb{R}}

\newcommand{\lef}{\left(}
\newcommand{\rig}{\right)}
\newcommand{\ilmanen}{\lef \R^{n+1}, e^{\frac{2}{n}x_{n+1}}\delta_{ij}\rig}

\newcommand{\mcv}{\boldsymbol{H}}
\newcommand{\un}{\nu}
\newcommand{\sbase}{e}
\newcommand{\oframe}{E}

\DeclareMathOperator{\conv}{Conv}

\newcommand{\fracsm}[2]{\begin{matrix}\frac{#1}{#2}\end{matrix}}

\newcommand{\beq}{\begin{equation}}
\newcommand{\eeq}{\end{equation}}

\newcommand{\spand}{{\operatorname{span}}}

\newcommand{\subRmany}{_{\R^{n+1}}\hspace{-1pt}}
\newcommand{\supRmany}{^{\R^{n+1}}\hspace{-1pt}}
\newcommand{\Rfew}{\R^{n}}
\newcommand{\Rmany}{\R^{n+1}}
\DeclareMathOperator{\tr}{tr}
\DeclareMathOperator{\dist}{dist}
\DeclareMathOperator{\Hess}{Hess}

\begin{document} 

\date{\today}

\title[bi-halfspace and convex hull theorems]{bi-halfspace and convex hull\\theorems for translating solitons}

\author{Francesco Chini\\}
\author{\\Niels Martin M\o{}ller}
\address{Francesco Chini, Department of Mathematical Sciences, Copenhagen University.}
\email{chini@math.ku.dk}

\address{Niels Martin M\o{}ller, Department of Mathematical Sciences, Copenhagen University.}
\email{nmoller@math.ku.dk}
\thanks{Francesco Chini was partially supported by the Villum Foundation's QMATH Centre. Niels Martin M\o{}ller was partially supported by the DFF Sapere Aude-Research Leader grant from The Independent Research Fund Denmark (Danish Ministry of Science), and by the United States National Science Foundation grant DMS-1311795.}
\keywords{Mean curvature flow, curvature flows, self-translaters, self-similarity, solitons, minimal surfaces, Omori-Yau principle, maximum principles, nonlinear partial differential equations.}

\begin{abstract}
While it is well known from examples that no interesting ``halfspace theorem'' holds for properly immersed  $n$-dimensional self-translating mean curvature flow solitons in Euclidean space $\mathbb{R}^{n+1}$, we show that they must all obey a general ``bi-halfspace theorem'' (aka ``wedge theorem''): Two transverse vertical halfspaces can never contain the same such hypersurface.  The same holds for any infinite end.  The proofs avoid the typical methods of nonlinear barrier construction for the approach via distance functions and the Omori-Yau maximum principle.

As an application we classify the convex hulls of all properly immersed  (possibly with compact boundary) $n$-dimensional mean curvature flow self-translating solitons $\Sigma^n$ in $\R^{n+1}$, up to an orthogonal projection in the direction of translation. This list is short, coinciding with the one given by Hoffman-Meeks in 1989, for minimal submanifolds: All of $\R^{n}$, halfspaces, slabs, hyperplanes and convex compacts in $\R^{n}$.

\end{abstract}

\maketitle
\section{Introduction}
The mean curvature flow for hypersurfaces in Euclidean space has been studied systematically since the late 1970s (to name but a few, see \cite{temam}, \cite{brakke}, \cite{huisken}, \cite{gage-hamilton}, \cite{grayson}, \cite{hamilton95}, \cite{white}, \cite{cm-mcf}, \cite{cm12}, and for early work on curve shortening flow \cite{mullins}), with considerable emphasis on the singularity models for the flow: the self-similar solitons.  

The oldest known nontrivial complete embedded soliton is Calabi's self-translating curve in $\R^2$, also sometimes called the ``grim reaper'' translating soliton (see Grayson \cite{grayson} and also \cite{mullins}, where it seems to have been first found). For readers more familiar with the Ricci flow, the most analogous object there would be Hamilton's cigar soliton (see \cite{hamilton}, and recall G. Perelman's central ``no cigar'' theorem \cite{perelman}).


Self-translaters arise in the study of the so-called  ``Type II'' singularities of the mean curvature flow. Indeed, using a classical result of Hamilton contained in \cite{hamilton95}, Huisken and Sinestrari \cite{HS99a} showed that   blow-up limit flows at Type II singularities of mean convex mean curvature flows are complete, self-translaters of the kind $\R^{n-k} \times \Sigma^k$, where $\Sigma^k$ is a convex translater in $\R^{k+1}$, with $k = 1, \dots, n$. For the mean convex case see also  \cite{HS99b}, \cite{Wh00}, \cite{Wh03} and \cite{HK17}.
If we remove the mean convexity hypothesis, it is known that blow-ups at Type II singularities must be eternal flows, but, to our knowledge, it is still not known whether these eternal flows are generally self-translaters.
(See Chapter 4 in \cite{mantegazza}.)


In the classical subject of minimal surfaces one of the cornerstones of the modern theory is the so-called ``Halfspace Theorem'' and convex hull classification, proven in 1989 by Hoffman and Meeks \cite{hoffman-meeks}. Numerous other authors have written about such halfspace theorems and convex hull properties, in various contexts: See f.ex. \cite{Xa84}, \cite{meeks-rosenberg}, \cite{bess}, \cite{meeks-rosenberg-again}, \cite{haus}, \cite{earp} and \cite{ro-schu-spru}.

In the literature, there are some results at the intersection of these two topics, of solitons and halfspace theorems. For instance in \cite{wei_wylie} (see also \cite{petersen}) there are some results for $f$-minimal hypersurfaces for the case of $\Ric_f>0$, including a halfspace theorem for one important class of mean curvature solitons, the self-shrinkers (see also \cite{PR14}).  The paper \cite{CE16} also showed a halfspace theorem (by using the half-catenoid-like ``self-shrinking trumpets'' from \cite{steve-niels} as barriers) and \cite{impiri} showed a ``Frankel property'' for self-shrinkers (meaning: when it so happens that all minimal surfaces in a space must intersect, as in \cite{frankel} and \cite{petersen}). Additionally, for self-translaters, a few significant geometric classification and nonexistence results are now known, see \cite{xj_wang}, \cite{sha}, \cite{halihaj}, \cite{niels}, \cite{hasl}, \cite{jpg}, \cite{imp_rim}, \cite{bueno} and \cite{himw}, but these do not directly address the question of (bi-)halfspace and convex hull properties.

One good reason for the lack of results with a (bi-)halfspace theorem flavor in the case of self-translaters would likely be that the most naive results one might imagine are wrong: F.ex. vertical planes and grim reaper cylinders readily coexist as self-translating solitons without ever intersecting, so there is no easy general ``halfspace theorem'' nor any ``Frankel property''.
 Moreover the typical arguments employed often rely on constructing barriers. As discussed in the Appendix, a strategy using other exact solutions to the translater equation does not seem readily available here, except in the case of 2-dimensional surfaces in $\R^3$.

In the present paper we will present the following three main contributions on $n$-dimensional mean curvature flow self-translating solitons (also known as ``translaters'', ``self-translaters'', ``translators'' or ``self-translators'') in $\R^{n+1}$. We assume in the below that the translation direction is $e_{n+1}$.

\begin{theorem}[Bi-Halfspace Theorem]\label{bi-halfspace}
There does not exist any properly immersed self-translating $n$-dimensional hypersurface $\Sigma^n\subseteq\Rmany$,  without boundary, which is contained in two transverse vertical halfspaces of $\Rmany$.
\end{theorem}

\begin{theorem}[Bi-Halfspace Theorem w/ Compact Boundary]\label{bi-halfspace_boundary}
Suppose a properly immersed  connected self-translating $n$-dimensional hypersurface $(\Sigma^n,\partial\Sigma)$ in $\Rmany$ is contained in two transverse vertical halfspaces of $\Rmany$. If $\partial \Sigma$ is compact then $\Sigma$ is compact. 
\end{theorem}

In the next theorem we let $\pi \colon \R^{n+1} \to \R^n$ be the projection in the direction of translation
$\pi\lef x_1, \dots, x_n, x_{n+1} \rig = ( x_1, \dots, x_n)$.

\begin{theorem}[Convex Hull Classification]\label{convex_hull_noncomp_trans}
Let $(\Sigma^n,\partial \Sigma)$ be a properly immersed connected self-translater in $\R^{n+1}$,  with (possibly empty) compact boundary $\partial \Sigma$.

Then exactly one of the following holds.
\begin{enumerate}
\item $\conv (\pi( \Sigma)) = \R^n$,
\item $\conv (\pi ( \Sigma ) )$ is a halfspace of $\R^n$,
\item $\conv (\pi( \Sigma)) $ is a closed slab between two parallel hyperplanes  of $ \R^n$,
\item $\conv (\pi( \Sigma)) $ is a hyperplane in $ \R^n$,
\item $\conv (\pi( \Sigma)) $ is a compact convex set. This case occurs precisely when $\Sigma$ is compact.
\end{enumerate}
\end{theorem}

\begin{remark}
From examples (see below) there appears to be no hope of classifying any of the likely wild classes $\Sigma$, $\conv(\Sigma)$ or $\pi(\Sigma)$: Only after applying \emph{both} of the forgetful operations $\conv(\cdot)$ and $\pi(\cdot)$ do we find a short list, which in fact can be thought of plainly as ``vertical slabs'' (including their three degenerate cases).

Note also that  $\conv(\cdot)$ and $\pi(\cdot)$ can be freely switched in the statement of Theorem \ref{convex_hull_noncomp_trans}, because for any subset $\Omega \subseteq \R^{n+1}$ they commute:
$$
\conv\left( \pi \left( \Omega \right) \right) =\pi\left( \conv \left( \Omega \right) \right).
$$
\end{remark}

\begin{remark}
We note that each of the five cases of Theorem \ref{convex_hull_noncomp_trans} can happen, when $n\geq 2$, except possibly for Case (2). Leaving the case $n=1$ to the reader, let us list examples for each case, assuming $n\geq 2$ (see also the longer list of examples below at the end of Section \ref{sec:prelims}):
\begin{enumerate}

\item Take any rotationally symmetric $\Sigma^n$, e.g. the ``bowl'' translater.

\item  No examples appear to be known.
\item Take as $\Sigma^n$ a grim reaper cylinder or any in Ilmanen's $\Delta$-wing family.
\item Take as $\Sigma^n$ any vertical hyperplane of $\R^{n+1}$.
\item Take any compact subset of any of the known examples.
\end{enumerate}

\end{remark}



Observe that an immediate consequence of Theorem \ref{bi-halfspace_boundary} is  the following

\begin{corollary}\label{corollary_ends}(Ends)
Any  end of a properly immersed self-transating $n$-dimensional hypersurface $\Sigma $ cannot be contained in two transverse vertical halfspaces of $\R^{n+1}$. 
\end{corollary}

\begin{remark}
The compact boundary version in Theorem \ref{bi-halfspace_boundary} does not follow from any generally valid modification of the proof of Theorem \ref{bi-halfspace}: For other related ambient spaces it can happen that even a halfspace theorem is true and yet no bi-halfspace theorem holds for the compact boundary case. See f.ex. the halfspace theorem for self-shrinkers in \cite{CE16}, and note how the asymptotically conical self-shrinkers in \cite{steve-niels} can easily be cut to get such examples which are noncompact with compact boundary.
\end{remark}

Let us quickly note how this is (for $\partial\Sigma=\emptyset$) strictly stronger than the old Hoffman-Meeks result, so that in the process we get a new proof of this classical fact:

\begin{corollary}[Hoffman-Meeks: \cite{hoffman-meeks}]\label{corr-hm}
The classification (1)-(5) in Hoffman-Meeks's Theorem 2 (Theorem \ref{hoffman_meeks_theorem} below) holds true for properly immersed minimal hypersurfaces in $\Rmany$ without boundary.
\end{corollary}
\begin{proof}[Proof of Corollary \ref{corr-hm}]
For $n\geq 2$, let $N^{n-1}\subseteq \Rfew$ be a  connected properly immersed minimal hypersurface. If $\partial N=\emptyset$, apply Theorem \ref{convex_hull_noncomp_trans} to the self-translater $\Sigma^n=N^{n-1}\times\Reals$. Then note
\[
\conv(N^{n-1})=\conv(\pi(N^{n-1}\times\Reals))=\conv(\pi(\Sigma)),
\]
from which the conclusion follows.
\end{proof}

As immediate corollaries to Theorem \ref{convex_hull_noncomp_trans}, we also recover the following previously known result:

\begin{corollary}[Corollary 2.2 \cite{xj_wang}]\label{corollary_domain_convex_graphs}
Let $\Sigma^n \subseteq \R^{n+1} $ be a complete connected convex graphical self-translater.  I.e.  there exists a smooth function $u~\colon~\Omega~\to~\R$, where $\Omega \subseteq \R^n$, such that $\graph\lef u \rig = \Sigma$.

Then exactly one of the following holds.
\begin{enumerate}
\item $\Omega = \R^n$.
\item $ \Omega$ is a halfspace in $\R^n$.
\item $\Omega$ is a slab between two parallel hyperplanes of $\R^n$.
\end{enumerate}
\end{corollary}

\begin{proof}
Since $\Sigma $ is convex and complete, from a theorem of Sacksteder (see \cite{sacksteder}), we have that $\Sigma = \partial C$, where $C \subseteq \R^{n+1}$ is a convex set. Therefore $\Sigma$ is a closed set w.r.t. the ambient topology and thus is properly embedded.

Let $u \colon \Omega \subseteq \R^{n} \to \R$ be a smooth function such that $\Sigma = \graph(u)$. Then clearly $\Omega$ is convex (indeed it is the orthogonal projection of the convex set $C$ onto $\R^n$) and $u$ is a convex function. Therefore
$$
\conv ( \pi (\Sigma)) = \conv( \Omega) = \Omega.
$$
We can now apply Theorem \ref{convex_hull_noncomp_trans} in order to conclude the proof.
\end{proof}

\begin{remark}
X.-J. Wang proved more than Corollary \ref{corollary_domain_convex_graphs}: For convex graphs, Case $(2)$ (graph over a halfspace) cannot happen. 
\end{remark}

In \cite{SX17}, Spruck and Xiao showed that any complete oriented immersed mean convex $2$-dimensional self-translater is convex. In particular, any complete $2$-dimensional graphical self-translater is convex.  Therefore in the case $n = 2$ one can improve Corollary \ref{corollary_domain_convex_graphs} removing the convexity assumption. In particular we recover the following result. 

\begin{corollary}[\cite{himw} and \cite{SX17}]
\label{corollary_no_wedge_domain} The domains for 2-dimensional graphical self-translaters belong to the Cases (1)-(3), respectively all $\R^2$, half-planes or slabs in $\R^2$. In particular, a properly immersed  self-translating $2$-dimensional hypersurface $\Sigma^2\subseteq \Reals^3$ cannot be the graph over a wedge-shaped domain in $\Reals^2$.
\end{corollary}

\begin{remark}
The above Corollary \ref{corollary_no_wedge_domain} is contained in the paper \cite{himw}, where all complete $2$-dimensional graphical self-translaters have very recently been fully classified (using \cite{SX17}). Again, Case (2) in fact cannot happen for 2-dimensional graphs.
\end{remark}


In \cite{sha} and \cite{sha15}, Shahriyari proved that there are no complete $2$-dimensional translaters which are graphical over a bounded domain. This fact was later generalized by M\o{}ller in \cite{niels} (see \cite{halihaj} for the half-cylinder case), where he proved that there are no properly embedded  without boundary $n$-dimensional self-translaters contained in a cylinder of the kind $\Omega \times \R$, where $\Omega \subseteq \R^n$ is bounded:

\begin{corollary}[\cite{niels}]\label{corollary_cylinders}
No noncompact properly immersed self-translating $n$-dimensional hypersurface $(\Sigma^n, \partial \Sigma) $ in $  \R^{n+1}$  with compact boundary can be contained in a  cylinder $\Omega \times\Reals$ with $\Omega \subseteq \R^n$ bounded. 
\end{corollary}

\begin{proof}
The proof follows easily from Theorem \ref{bi-halfspace_boundary}.  Indeed note that given a bounded set $\Omega \subseteq \R^n$, the cylinder $\Omega \times\Reals$ is contained in the intersection of two transverse vertical halfspaces.
\end{proof}

\begin{remark}
The proof shows more than Corollary \ref{corollary_cylinders}, namely that the conclusion holds assuming only boundedness in two directions: $\Sigma^n\subseteq\Omega_2\times\Reals^{n-1}$ cannot happen for $\Omega_2\subseteq\Reals^{2}$.
\end{remark}

As will be clear below, most of the ideas that we will need were essentially in place as early as the 1960s, much earlier than the minimal surface and curvature flow papers cited above. Namely, in the original paper by Omori \cite{omori}, he showed by quite similar methods that in Euclidean $n$-space,  cones with angle $0 < \theta < \pi$ cannot contain  properly embedded minimal surfaces.

Somewhat later, in 1989, contained within the proof of ``Theorem 2'' from \cite{hoffman-meeks} (which seems independent of Omori's ideas) is the fact that, while the Hoffman-Meeks ``halfspace theorem'' only works for minimal 2-surface immersions $\Sigma^2\to\Reals^3$, one has a ``bi-halfspace theorem'' (stronger than the cone theorems) for minimal hypersurfaces $\Sigma^{n}\to\Rmany$ for $n\geq 3$, even allowing compact boundary. Their proof used barriers from the nonlinear Dirichlet problem known as the $n$-dimensional Plateau problem for graphs. Some disadvantages of that approach are clear: For when do such barriers exist, and if they in fact do, what are their precise properties, as needed for a ``separating tangency'' argument to run?

It then appears that only within the last decade it was realized by Borb\'ely \cite{borbelywedge} that one can prove bi-halfspace theorems for minimal 2-surface immersions $\Sigma^2\to\Reals^3$, under the assumption that the Omori-Yau principle (so named after \cite{omori}-\cite{cheng-yau}) is known to be available on the given $\Sigma^2$. This was also expanded by Bessa, de Lira and Medeiros in \cite{bessalira} where they showed Borb\'ely-style ``wedge'' theorems for stochastically complete  minimal surfaces in Riemannian products $(M\times N, g_M\oplus g_N)$, where $(N,g_N)$ is complete without boundary.  Seeing as the Huisken-Ilmanen metric, in which self-translaters are the minimal surfaces, is not a Riemannian product\footnote{Note however that \cite{smocz} showed that it can be seen as a warped Riemannian product.} nor complete, and our surfaces can have boundaries, 
we will directly take Borb\'ely's method as our point of departure.

Here, in our case of $n$-dimensional self-translaters $\Sigma^n\to\Rmany$, the Omori-Yau maximum principle in turn works quite generally, which is a well-established fact that has previously been invoked by several authors for related problems: See \cite{xin}, \cite{SX16}-\cite{SX17} and \cite{imp_rim}. Many other authors have written on the topic, see e.g. \cite{schoen-yau-lectures}, \cite{pigola03}, \cite{barr}. For a general yet particularly easy to state result, let us mention this: The Omori-Yau maximum principle holds for every
submanifold properly immersed with bounded mean curvature into a Riemannian space form (see \cite{PRS05}). Here we will be using the formulation and short proof in \cite{xin}, so as to make the whole presentation quite elementary and essentially self-contained, including as a biproduct the proof of the Hoffman-Meeks results for $n\geq 3$ and empty boundary, in Corollary \ref{corr-hm} below.

In a later work \cite{CM19}, we generalize the main ideas contained in the present paper to ancient mean curvature flows, providing  a parabolic Omori-Yau principle  and using it for proving a bi-halfspace theorem for ancient flows.

\section{Overview}\label{sec:outline}
In Section \ref{sec:prelims} we introduce notation and list a few of the technical lemmas in the form that we will need them later, with (references to) short proofs.

In Section \ref{sec:bi-halfspace} we prove a new ``Bi-Halfspace Theorem'' for  properly immersed self-translaters, which is Theorem \ref{bi-halfspace}. We also fully classify all the possible pairs of halfspaces such that their intersections contain a complete self-translater, in Corollary \ref{bi-halfspace_corollary}.

In Section \ref{sec:basics} we study the convex hull of such hypersurfaces, both for compact self-translaters and for noncompact ones, but with compact (possibly empty) boundary. We observe a behavior very similar to the one of minimal submanifolds of the Euclidean space.  The main result of the section is Theorem~\ref{convex_hull_noncomp_trans} and it was inspired by a result by Hoffman and Meeks  in the context of minimal submanifolds of $\R^{n+1}$ (see  \cite{hoffman-meeks}). The proof here is based on our ``Bi-Halfspace'' Theorem \ref{bi-halfspace} and the compact boundary version Theorem \ref{bi-halfspace_boundary} and hence diverges significantly from the proof of the theorem of Hoffman and Meeks, which relied on constructing barriers via certain nonlinear Dirichlet problems.

In the Appendix (Section \ref{appendix}) we will comment more on this point and we will provide an alternative proof of Theorem \ref{convex_hull_noncomp_trans}, which is closer in spirit to the one by Hoffman and Meeks, but which only works in the case  $n=2$.

\section{Preliminaries and Notation}\label{sec:prelims}

In what follows, $(x_1,x_2, \dots, x_{n}, x_{n+1})$ are the standard coordinates of $\R^{n+1}$ and  $(\sbase_1,\sbase_2, \dots, \sbase_n, \sbase_{n+1})$ is the standard orthonormal basis of $\R^{n+1}$.

 On $\R^{n+1}$ we will, with a slight abuse of notation, denote the coordinate vector fields by $\partial_i=\frac{\partial}{\partial x_i}=e_i$. 


In this paper $\Sigma^n \subseteq \R^{n+1}$ will always denote a smooth properly immersed self-translater with velocity vector $\sbase_{n+1}$. 
Recall that properly immersed hypersurfaces with boundary are geodesically complete with boundary in the induced Riemannian metric (the Heine-Borel property with Hopf-Rinow).

The evolution of $\Sigma^n$  under the mean curvature flow is a  unit speed translation in the direction of the positive $x_{n+1}$-axis. Therefore $\Sigma^n$ satisfies the following equation
\begin{equation}\label{translater_equation}
\mcv =  \langle \sbase_{n+1}, \un \rangle \un,
\end{equation}
where $\mcv=-H\nu$ is the mean curvature vector of $\Sigma^n$ and $\un$ is the unit normal vector field on $\Sigma^n$.

Let us recall here two important tools that we will need for our work.  

\begin{lemma}[Comparison Principle for MCF]\label{comparison_principle}
Let $\varphi \colon M_1 \times [0, T) \to \Rmany $ and $\psi \colon M_2 \times [0, T)\to \Rmany$ be two hypersurfaces evolving by mean curvature flow and let us assume that $M_1$ is properly immersed while $M_2$ is compact. Then the distance between them is nondecreasing in time.
\end{lemma}

\begin{proof}
See e.g. the proof of Theorem 2.2.1 in \cite{mantegazza}.
\end{proof}

\begin{lemma}[Principle of Separating Tangency for  Self-Translaters] \label{tangency_principle}
Let $\Sigma_1^n$ and $\Sigma_2^n$ be two connected (unit speed, same direction) self-translaters  immersed into $ \R^{n+1} $, with (possibly empty) boundaries $\partial \Sigma_1$ and $\partial \Sigma_2$.

Suppose that there exists a  point $p \in \Sigma_1 \cap \Sigma_2$ such that it is an interior point for both the self-translaters. Let us assume that the corresponding tangent spaces $T_{p}\Sigma_1 $ and $T_p\Sigma_2$ coincide and assume that, locally around $p$, $\Sigma_1$ lies on one side of $\Sigma_2$. 

Then there are open neighborhoods $U_1 \subseteq \Sigma_1$ and  $U_2 \subseteq \Sigma_2$ of $p$ such that $U_1 = U_2$.
\end{lemma}

\begin{proof}
This uses the maximum principle and unique continuation. See Theorem 2.1.1 in \cite{jpg}, Lemma 2.4 in \cite{niels} and Theorem 2.1 in \cite{mss}.
\end{proof}


\subsection*{Well-known Examples}
We conclude this section by enumerating some of the most well-known examples of self-translaters.
\begin{enumerate}
\item (Translating minimal hypersurfaces) Any hyperplane of $\R^{n+1}$ which is parallel to $e_{n+1}$ is a self-translater. More generally, if $N^{n-1} \subseteq \R^n$ is a minimal submanifold, then we have that $\Sigma \coloneqq N \times \R \subseteq \R^{n+1}$ is self-translating in the $e_{n+1}$-direction. This follows from the short computation $H_{N\times \R}=0=\left\langle(\nu_N,0),(\textbf{0},1)\right\rangle_{\Rmany}$.
\item (Grim reaper cylinder) Consider the function $f \colon \lef -\frac{\pi}{2}, \frac{\pi}{2}\rig \to \R$ defined as $f(x) \coloneqq - \ln\lef \cos \lef x \rig \rig $. Its graph $\Gamma \coloneqq \graph\lef f \rig$ is called Calabi's \emph{grim reaper} curve (first found in \cite{mullins}) and it is the only nonflat connected complete translating soliton for the curve shortening flow. The hypersurface $\Gamma^n \coloneqq \R^{n-1} \times \Gamma \subseteq \R^{n+1}$ is called a \emph{grim reaper cylinder} and it is a self-translater.

\item (Rotationally symmetric self-translaters) In \cite{CSS}, the authors classify all the self-translaters which are rotationally symmetric with respect to the $x_{n+1}$-axis. These are the so-called \emph{bowl soliton} $U$ which was already discovered in \cite{aw}, and the family of \emph{winglike self-translaters}, also known as \emph{translating catenoids}.  
The bowl soliton is the graph of an entire convex function $u \colon \R^n \to \R$ and it is asymptotic to a paraboloid. Indeed it is also known as the \emph{translating paraboloid}. 

The wing-like self-translaters are all diffeomorphic to $\mathbb{S}^{n-1} \times \R$, where $\mathbb{S}^{n-1}$ is the $(n-1)$-dimensional sphere. They roughly look like two bowl solitons, one above the other, glued together with a vertical neck. Both of the ends are asymptotic to $U$. For each $R > 0$ there exists a unique (up to a translation in the $x_{n+1}$ direction) winglike self-translater $W_R$ such that the size of its neck is $R>0$. 

\item (Gluing constructions) The desingularization techniques, originally  developed by Kapouleas (see \cite{kap}) for building new examples of minimal and constant mean curvature hypersurfaces,  have been applied  by X.H. Nguyen and others, in order to prove the existence of new translating solitons, by ``gluing together'' already known examples. For more details, we refer to \cite{tridents}, \cite{finite_gr}, \cite{doubly_periodic},  \cite{ddn} and \cite{smith}. See also \cite{kkm} (and \cite{nguy-shrink}) for the first gluing construction for mean curvature solitons with non-flat ends.

\item (Delta-wing self-translaters) Recently,  Bourni, Langford, and Tinaglia (Theorem 1 in \cite{blt18}), and independently Hoffman, Ilmanen, Mart\'in and White (Theorems 4.1, 8.1 in \cite{himw}) have proved that for each $b > \frac{\pi}{2}$, there exists a strictly convex and complete self-translater which lies in the slab $(-b, b) \times \R^n$ and in no smaller slab.

Furthermore, also uniqueness was proven in \cite{himw}. They called this new family of self-translaters, which is parametrized by the width of the slab, the $\Delta$-wings. 

\item (Annuli, helicoid and Scherk's) In an upcoming paper \cite{himw2}, the authors have announced that they will be constructing several new families of  properly embedded (nongraphical)
translators (quoting the abstract for a talk at Stanford in July 2018): ``[...] a two-parameter family of translating annuli, examples that
resemble Scherk’s minimal surfaces, and examples that resemble helicoids.''
\end{enumerate}


\section{Bi-halfspace Theorems for Self-Translating Solitons}\label{sec:bi-halfspace}

In this section we prove the ``Bi-Halfspace'' Theorem \ref{bi-halfspace} and the case with boundary Theorem \ref{bi-halfspace_boundary}. Let us first make a few remarks:

\begin{remark}
In the theorems, the transversality can simply be defined via the unit normals to the boundary hypersurfaces (which are affine hyperplanes) of the halfspaces: They must not be (anti-)parallel as vectors in $\Rmany$.

Note that these theorems are vacuously true for $n=1$, as in $\Reals^2$ all vertical affine halfspaces are (anti-)parallel and hence never transverse. Thus, in the below we will throughout tacitly assume $n\geq 2$.

Note also that the statements and proofs of the ``Bi-Halfspace'' Theorem \ref{bi-halfspace} and the case with boundary Theorem \ref{bi-halfspace_boundary} can be either false or true, with an easy proof, if one or both of the two halfspaces are not vertical. See Corollary \ref{bi-halfspace_corollary} at the end of this section for a clarification.
\end{remark}

Let us state the version of the Omori-Yau lemma which we will be needing:
\begin{lemma}(Omori-Yau for Translating Solitons)\label{mementomori}
Let $(\Sigma^n,\partial\Sigma)$ be a properly immersed self-translating soliton in $\Reals^{n+1}$ which is complete with boundary. Suppose that $f:\Sigma^n\to\R$ is a function which satisfies:
\begin{itemize}
\item[(i)] $\sup_\Sigma |f|<\infty,\quad \sup_{\partial \Sigma} f<\sup_\Sigma f$,
\item[(ii)]$f\in C^0(\Sigma)$,
\item[(iii)] $\exists \varepsilon_f>0$ s.t. $f$ is $C^2$ on the set $\{p\in\Sigma: f(p) > \sup_\Sigma f -\varepsilon_f\}$.
\end{itemize}
Then there exists a sequence $\{p_k\}$ in $\Sigma^n$ such that:
\begin{align}
&\lim_{k\to\infty} f(p_k)= \sup_{\Sigma} f,\label{OY_prop1}\\
&\lim_{k\to\infty} \nabla^\Sigma f(p_k)=0\label{OY_prop2},\\
&\lim_{k\to\infty} \Delta_\Sigma f(p_k)\leq 0.\label{OY_prop3}
\end{align}

\end{lemma}
\begin{proof}[Proof of Lemma \ref{mementomori}]
A short direct proof can be found in \cite{xin} (using that $\Sigma^n$ is complete with boundary and properly immersed), which is easily adapted to the form stated here. For bounded $|f|$ the condition of Xin,
\[
a_k\in\Sigma^n,\quad\|a_k\|\subRmany\to\infty\quad\Rightarrow\quad\lim_{k\to\infty}\frac{f(a_k)}{\|a_k\|\subRmany}=0
\]
is of course trivially satisfied.
\end{proof}

\begin{proof}[Proof of the ``Bi-Halfspace'' Theorem \ref{bi-halfspace}]
Any affine halfspace $H\subseteq\Reals^{n+1}$ can be given by a pair of (offset and direction, resp.) vectors $(b,w)\in \Reals^{n+1}\times \mathbb{S}^{n}$, where we view $\mathbb{S}^{n}\subseteq\Rmany$. Namely: 
\begin{align*}
&H=H_{(b,w)}:=\left\{x\in\Reals^{n+1}: \langle x-b,w\rangle\geq 0\right\},\\
&P:=\partial H = \left\{x\in\Reals^{n+1}: \langle x-b,w\rangle= 0\right\}.
\end{align*}
Note that $w$ is unique but any $b\in \partial H$ works. Recall that such two $n$-planes $P_1,P_2$ have transverse intersection $P_1\pitchfork P_2$ if and only if the corresponding unit normals $w_1\nparallel w_2$ (so antiparallel is also forbidden). This is also what it means for two halfspaces $H_1$ and $H_2$ to be transverse.

What we call vertical halfspaces are those $H_{(b,w)}$ for which $w\perp e_{n+1}$, i.e. $w=(w^{(1)},\ldots,w^{(n)},0)\in\mathbb{S}^{n}\times\{0\}$.

We now perform a couple of normalizations which are not essential but greatly simplify some of the computations: Suppose that an $e_{n+1}$-directed self-translating hypersurface $\Sigma^n\subseteq\Rmany$ is contained in a pair of transverse vertical halfspaces, i.e. that $\Sigma^n\subseteq H_1\cap H_2$. By simultaneously moving $\Sigma^n$ and $H_i$, we may assume $b_1=b_2=0$ (pick any $b\in H_1\cap H_2$, then translate by $-b$). Note also that $\spand (w_1,w_2)$ defines a $2$-dimensional subspace in $\Rfew\times\{0\}$.

We can then, by acting rigidly with $O(n)$ on the $\Reals^{n}$-factor (take an orthonormal basis for this 2-plane, fill out to an orthonormal basis of $\Rfew$ finally compose with an $O(2)$-map in the two first coordinates), we can assume that there exists $(\xi,\eta)$ such that $\xi,\eta>0$ with $\|(\xi,\eta)\|=1$ and:
\[
w_1 = (\xi,\eta,0,\ldots,0),\quad w_2 = (\xi,-\eta,0,\ldots,0).
\]

As explained in the introduction, we will now proceed with an adaptation of the method of Borb\'ely to our situation of $n$-dimensional self-translaters. Consider for $R>0$ the respective affine hyperplanes of equidistance: $P_i + Rw_i=\{x: \langle x,w_i\rangle = R\}$. Their intersection locus is an $(n-1)$-dimensional vertical affine subspace $\mathscr{L}_R:=(P_1 + Rw_1)\cap(P_2 + Rw_2)$. Linear algebra reveals a simple explicit expression for this locus:
\beq\label{LocusElectus}
\mathscr{L}_R:=\left\{\left(\fracsm{R}{\xi},0,x_3,\ldots,x_{n+1}\right):\: (x_3,\ldots,x_{n+1})\in\Reals^{n-1}\right\}.
\eeq

We consider then the ambient Euclidean distance function from points $x\in\Rmany$ to $\mathscr{L}_R$:
\beq
d(x):=d_R(x):=\dist\subRmany(x,\mathscr{L}_R)=\sqrt{\left(x_1-\fracsm{R}{\xi}\right)^2 + x_2^2},\quad x\in\Rmany.
\eeq

Clearly $\mathscr{L}_R=\{x\in\Rmany: d_R(x) = 0\}$ and $\|\nabla\supRmany d\|=1$ on $\Rmany\setminus\mathscr{L}_R$. We define the cylindrical set by:
\[
\mathscr{D}_R=\left\{x\in\Rmany: d_R(x) \leq R\right\},
\]
which is an $(n+1)$-dimensional solid with boundary.
Then for any $R>0$, explicitly
\[
\mathscr{D}_R\cap P_i = \left\{\left(\fracsm{R\eta^2}{\xi},(-1)^{i} R\eta,x_3,\ldots,x_{n+1}\right):  (x_3,\ldots,x_{n+1})\in\Reals^{n-1}\right\},
\]
which disconnects $\partial{(H_1\cap H_2)}$ and the set $(H_1\cap H_2)\setminus\mathscr{D}_R$ has exactly two connected components (both unbounded).

We label by $\mathcal{V}_R$ the connected component of $(H_1\cap H_2)\setminus\mathscr{D}_R$ where $d_R$ is bounded (the other component, where $d_R$ is unbounded, we will not need to refer to directly). Notice that as $R\nearrow \infty$ we have $\mathcal{V}_R\nearrow H_1\cap H_2$. From now on, we will pick a fixed $R>0$ large enough so that $\Sigma\cap \mathcal{V}_R\neq \emptyset$.

In the below, we will at times drop the subscript and write $d(x):=d_R(x)$.

A couple of standard, elementary computations show that
\begin{align}
\label{ZeroDir}
\Hess\subRmany d \left(\nabla\supRmany d_R,\nabla\supRmany d_R \right)&=0, \quad\textrm{on}\quad \Rmany\setminus \mathscr{L}_R,\\
\Delta\subRmany d_R & = \frac{1}{d_R}, \quad\textrm{on}\quad \Rmany\setminus \mathscr{L}_R.
\end{align}
The first equation, giving an eigenvector field for the eigenvalue $\lambda = 0$, can also be deduced from $d_R(x)$ being linear in the gradient direction. Note also that as $d_R$ does not depend on the last $n-1$ coordinates of $\Rmany$, $\Hess\subRmany$ has the  $n-1$ orthonormal eigenvector fields with eigenvalue zero $e_3,\ldots,e_{n+1}$, all perpendicular to $\nabla\supRmany d_R$. The only nonzero eigenvalue is $\lambda = 1/d_R$ with unit length eigenvector field correspondingly given by e.g.
\beq\label{ChiField}
\chi=\Big(-\fracsm{\partial d_R}{\partial x_2},\fracsm{\partial d_R}{\partial x_1},0,\ldots,0\Big), \quad\textrm{on}\quad \Rmany\setminus \mathscr{L}_R,
\eeq
which together with the other listed eigenvector fields forms an orthonormal frame field on $\Reals^{n+1}\setminus\mathcal{L}_R$.

The following simple fact follows from a small exercise in linear algebra: Given a square symmetric matrix $A\in\mathrm{Mat}_{n+1}(\Reals)$ the trace over an $n$-dimensional hyperplane $P_{\mu}$ defined by a unit normal vector $\mu\in\Rmany$ is:
\beq
\tr_{\mu}(A) = \sum_{i=1}^{n+1}\lambda_i\left(1-\left(\left\langle v_i,\mu\right\rangle\subRmany\right)^2\right),
\eeq
where the $(\lambda_1,\ldots,\lambda_{n+1})$ are the eigenvalues of $A$ with multiplicity and $(v_i)\subseteq\Rmany$ a corresponding orthonormal basis of eigenvectors. Thus in our case of a Hessian with only one nonzero eigenvalue and corresponding unit eigenvector field $\chi$, we get the comparatively simple expression from tracing over $T_p\Sigma$ with the unit normal $\nu$:
\beq\label{TraceTrace}
\tr_{\Sigma} \left(\Hess\subRmany d\right) = \frac{1-\left(\left\langle \chi,\nu\right\rangle\subRmany\right)^2}{d},\quad\textrm{on}\quad\Sigma^n\setminus \mathscr{L}_R.
\eeq

We now define the modified distance function $f:\Sigma^n\to \Reals$:
\begin{equation}\label{definition_f}
f(p)=\begin{cases}
d_R(p), \quad p\in\Sigma\cap \mathcal{V}_R,\\
R, \quad\quad\:\:\: p\in\Sigma^n\setminus\big(\mathcal{V}_R\cap\mathscr{D}_R\big).
\end{cases}
\end{equation}
This function is well-defined and continuous (as $d_{\mid\partial\mathscr{D}_R}=R$) and it is smooth on $\Sigma^n\setminus\mathscr{D}_R$. It is also bounded, namely note that explicitly we have (using for the first inequality that $R>0$ was fixed large enough that $\Sigma\cap \mathcal{V}_R\neq\emptyset$, and recall also $0<\xi<1$):
\beq\label{f-bounds}
R<\sup_\Sigma f \leq R/\xi<\infty.
\eeq

At points $p\in\Sigma\cap \mathcal{V}_R$ (so that in particular $f=d_{\mid\Sigma}$ is smooth), we have that the gradient equals the tangential part of the ambient gradient:
\beq\label{grad_f}
\nabla^\Sigma f = \left(\nabla\supRmany d\right)^\top = \nabla\supRmany d- \left(\nabla\supRmany d\right)^\perp=\nabla\supRmany d - \langle\nabla\supRmany d,\nu\rangle\subRmany\nu,
\eeq
with length computed using (\ref{ChiField}) to be (recall again $\|\nabla\supRmany d\|\subRmany=1$):
\beq\label{dot_to_zero}
\begin{split}
\|\nabla^\Sigma f\| & = \sqrt{1 - \left(\left\langle \nabla\supRmany d,\nu\right\rangle\subRmany\right)^2}\\
& = \left|\left\langle \chi,\nu\right\rangle\subRmany\right|.
\end{split}
\eeq
So we can finally recast (\ref{TraceTrace}) as the following fundamental identity for the distance function to the locus $\mathcal{L}_R$:
\beq\label{Laplace-f-one}
\tr_{\Sigma} \left(\Hess\subRmany d_R\right) = \left(1-\|\nabla^\Sigma f\|^2\right)\Delta\subRmany d_R,\quad\mathrm{on}\quad\Sigma\cap \mathcal{V}_R.
\eeq

We recall that the vector-valued second fundamental form is $A(X,Y):=(\nabla^{\Rmany}_X Y)^\perp$. Now apply (\ref{grad_f}) and recall $\nabla^\Sigma_X Z=\left(\nabla^{\Rmany}_X \overline{Z}\right)^\top$, for $\overline{Z}$ any extension of $Z$. Then for any $X,Y\in T_p\Sigma$:
\begin{align*}
\Hess_\Sigma f(X,Y) &:= \left\langle \nabla^\Sigma_X\nabla^\Sigma f, Y \right\rangle_\Sigma = 
\left\langle \nabla^\Sigma_X\left[\nabla\supRmany d - \left(\nabla\supRmany d\right)^\perp\right], Y \right\rangle\\
&=\left\langle \nabla_X\supRmany\left[\nabla\supRmany d - \overline{\left(\nabla\supRmany d\right)^\perp}\right], Y \right\rangle\\
&=\Hess\subRmany d(X,Y) - \left\langle \nabla_X\supRmany\overline{\left(\nabla\supRmany d\right)^\perp}, Y \right\rangle\\&=\Hess\subRmany d(X,Y) + \left\langle \nabla\supRmany d, A(X,Y)\right\rangle\subRmany,
\end{align*}
where the last step is seen by computing
\[
X.\left\langle \overline{\left(\nabla\supRmany d\right)^\perp}, \overline{Y} \right\rangle = \left\langle \nabla_X\supRmany\overline{\left(\nabla\supRmany d\right)^\perp}, \overline{Y} \right\rangle + \left\langle \overline{\left(\nabla\supRmany d\right)^\perp}, \nabla_X\supRmany \overline{Y} \right\rangle,
\]
and then evaluting on $\Sigma$ to get:
\[
0=\left\langle \nabla_X\supRmany\overline{\left(\nabla\supRmany d\right)^\perp}, Y \right\rangle + \left\langle \left(\nabla\supRmany d\right)^\perp, A(X,Y)\right\rangle.
\]

Taking now the trace over $T_p\Sigma$ we see:
\beq
\Delta_\Sigma f = \tr_{\Sigma} \left(\Hess\subRmany d\right) + \left\langle \nabla\supRmany d, \mcv\right\rangle\subRmany
\eeq

Here we used that the mean curvature vector is $\mcv:=\tr_\Sigma A = -H\nu$. Using now the self-translater equation $H=\langle e_{n+1},\nu\rangle$, we get:
\beq\label{Laplace-f}
\Delta_\Sigma f = \tr_{\Sigma} \left(\Hess\subRmany d\right) - \langle \nabla\supRmany d, \nu \rangle\langle e_{n+1},\nu\rangle.
\eeq

Combining (\ref{Laplace-f-one}) and (\ref{Laplace-f}) we finally have shown:
\beq\label{main_identity}
\Delta_\Sigma f=\frac{1-\|\nabla^\Sigma f\|^2}{d} - \big\langle \nabla\supRmany d, \nu \big\rangle\big\langle e_{n+1},\nu\big\rangle,\quad\mathrm{on}\quad \Sigma\cap \mathcal{V}_R.
\eeq

We will now apply the Omori-Yau principle in Lemma \ref{mementomori} to $f:\Sigma^n\to\R$, so we get a sequence of points $\{p_k\}$ on $\Sigma^n$ with the Omori-Yau properties (\ref{OY_prop1})-(\ref{OY_prop3}). To see that the Omori-Yau principle indeed applies here,  we check that all the conditions in Lemma \ref{mementomori} hold. By construction $0<\sup_{\Sigma} f<\infty$, $f\in C^0(\Sigma)$ and $f$ is $C^2$ where relevant. Recall also that since by (\ref{f-bounds}) we know $\sup_\Sigma f > R $, and as $f|_{\Sigma\setminus \mathcal{V}_R}\leq R$ (note also that in principle $\Sigma\setminus \mathcal{V}_R=\emptyset$ is possible), we may assume that all $p_k\in \Sigma\cap \mathcal{V}_R$.

To proceed we now need to analyze the last ``perturbation term'' in (\ref{main_identity}), which came from the self-translater equation. Notice first that by the triangle inequality 
\beq\label{SqueezeThatTerm}
\big|\big\langle e_{n+1},\nu\big\rangle\big|\leq \big|\big\langle e_{n+1},\nabla\supRmany d\big\rangle\big| + \big|\big\langle e_{n+1},\nu-\nabla\supRmany d\big\rangle\big|\leq\big\|\nu-\nabla\supRmany d\big\|,
\eeq
using also the fact that $\langle e_{n+1},\nabla\supRmany d\rangle=0$ and finally applying the Cauchy-Schwarz inequality.

We know from the property (\ref{OY_prop2}) combined with Equation (\ref{dot_to_zero}) that the limit
\beq\label{junk_limit}
\big|\big\langle \nabla\supRmany d, \nu \big\rangle\big|(p_k)\to 1,\quad\mathrm{as}\quad k\to\infty.
\eeq
holds, so from a certain stage the inner product has at each point a definite sign. By the Pigeon Hole Principle, there must then exist a sign $\sigma_\infty\in\{-1,1\}$ and a subsequence of points such that $\langle \nabla\supRmany d, \nu \rangle \to \sigma_\infty$. So by, if necessary, flipping orientations $\nu \leftrightarrow -\nu$ (a symmetry for the self-translater equation) we may assume that $0<\langle \nabla\supRmany d, \nu \rangle \to 1$ on the sequence of points. This also leads to:
\beq
\big\|\nu(p_k)-\nabla\supRmany d_R(p_k)\big\|_{\Rmany}^2 = 2\left[1 - \langle \nabla\supRmany d, \nu\rangle\right]\to 0.
\eeq
In consequence, we can use (\ref{SqueezeThatTerm}) to conclude that:
\beq\label{JunkTerm}
\left|\left\langle e_{n+1},\nu\right\rangle\right|(p_k)\to 0.
\eeq
Now, from (\ref{JunkTerm}) with either (\ref{junk_limit}) or simply $|\langle \nabla\supRmany d, \nu \rangle|\leq 1$, the last term in (\ref{Laplace-f}) tends to zero. Going to the limit in (\ref{main_identity}), we thus conclude that the limits exist in the following relation:
\beq\label{laplace_contradict}
\lim_{k\to\infty}\Delta_\Sigma f (p_k)= \lim_{k\to\infty}\frac{1}{d(p_k)} \geq \frac{\xi}{R} > 0,
\eeq
using again $0<\xi<1$.
This violates Property (\ref{OY_prop3}) in the Omori-Yau maximum principle of Lemma \ref{mementomori}, namely that $\lim_{k\to\infty} \Delta_\Sigma f(p_k)\leq 0$. This contradiction concludes the proof that there cannot exist any such self-translater.
\end{proof}

\begin{proof}[Proof of the Theorem \ref{bi-halfspace_boundary}]
To proceed in the case of compact nonempty boundary, we will again assume that $H_1$ and $H_2$ are as in the proof of the ``Bi-Halfspace'' Theorem \ref{bi-halfspace}, while we now allow $(\Sigma^n,\partial\Sigma)$ to be complete with compact boundary and still properly immersed. We furthermore assume that $\Sigma^n$ is connected. For every $R>0$, let $\mathscr{L}_R$, $\mathscr{D}_R$ and $d = d_R$ be as in the proof of the Theorem \ref{bi-halfspace}. Recall that $\mathcal{V}_R$ denotes that connected component of $\lef H_1 \cap H_2 \rig\setminus \mathscr{D}_R$ on which $d$ is bounded. Let again $f$ be the function defined in \eqref{definition_f}.
Note that since $\partial \Sigma$ is compact, we can pick $R>0$ large enough so that $\partial \Sigma \subseteq \mathcal{V}_R$.

We will now, for contradiction, assume that $(\Sigma,\partial \Sigma)$ is not compact. We will distinguish between two different cases and finally see that each of them leads to a contradiction. 

\begin{itemize}
\item \textbf{Case (a)}: $\Sigma\cap \mathcal{V}_R$ is bounded in $\Rmany$ for every $R > 0$.

\item \textbf{Case (b)}:  There exists $R > 0$ s.t. $\Sigma\cap \mathcal{V}_R$ is unbounded in $\Rmany$.
\end{itemize}

\noindent\textbf{Proof for Case (a)}: 
By the definition of $\mathscr{D}_R$, we can fix $R>0$ large enough so that
\begin{equation}\label{distance_between_components}
\dist\lef \partial \Sigma,  \mathscr{D}_R \rig > \pi.
\end{equation}

Since $\mathscr{D}_R\subseteq\Rmany$ has compact vertical projection, there exists an open vertical slab $S\subseteq\Rmany$ between two parallel vertical hyperplanes at distance $\pi$ separating $\partial \Sigma$ and $\mathscr{D}_R$ . More precisely, we can arrange that $\partial \Sigma$ and $\mathscr{D}_R$ are contained in two different connected components of $\R^{n+1} \setminus \overline{S}$. Let now $\Gamma^n:=\Gamma\times\R^{n-1} \subseteq S$  be a grim reaper cylinder.  Let us consider the family $\{ \Gamma^n_s \}_{s \in \R}$ defined via $\Gamma^n_s \coloneqq \Gamma^n + s e_{n+1}$. Note that $\cup_{s\in\R}\Gamma^n_s=S$.

Since in the present case, $\Sigma^n$ is assumed noncompact and hence unbounded (using that it is properly immersed), while $\Sigma\cap \mathcal{V}_R$ is assumed bounded, we surely have $\Sigma\setminus \mathcal{V}_R\neq\emptyset$ regardless of how large we take $R>0$. Seeing as $\Sigma^n$ is connected, we therefore conclude that $\Sigma\cap S\neq \emptyset$. Therefore there also exists $s\in\R$ small enough so that $\lef \Sigma\cap \mathcal{V}_R \rig \cap \Gamma^n_s\neq\emptyset$.

On the other hand, since $\Sigma\cap \mathcal{V}_R$ is assumed bounded, then for $s\in\R$ large enough we have that 
$
\lef \Sigma\cap \mathcal{V}_R \rig \cap \Gamma^n_s = \emptyset$. Because $\Gamma^n$ is properly embedded, and since $\Sigma\cap \mathcal{V}_R$ is assumed bounded, there exists an extremal value $s_0$:
$$
s_0 \coloneqq \sup \{ s \in \R \colon \lef \Sigma\cap \mathcal{V}_R \rig \cap \Gamma^n_s \ne \emptyset \}<\infty.
$$
By compactness of $\overline{\Sigma\cap \mathcal{V}_R}$ hence of $\overline{\Sigma\cap S}$ and since $\Sigma$ is properly immersed, this $s_0$ is attained at some $p_0 \in  \lef \Sigma\cap \mathcal{V}_R \rig \cap \Gamma^n_{s_0}$, where we note that $p_0 \in \overline{S}$. Therefore $p$ is a point of $\Sigma\cap \mathcal{V}_R$ which is interior relative to $\Sigma$. We can therefore apply Separating Tangency from Lemma \ref{tangency_principle}, which by completeness, connectedness and compactness of the boundary implies that $\Sigma$ and $\Gamma\times\R^{n-1}$ coincide outside some ambient ball, leading to a contradiction with f.ex. the assumption that $\Sigma\subseteq H_1\cap H_2$ (or with the boundedness of $\Sigma\cap \mathcal{V}_R$).

\noindent\textbf{Proof for Case (b)}: Let us summarize how we will now fix the setup throughout the rest of the proof: $R>0$ will be taken large enough so that $\partial \Sigma\subseteq \mathcal{V}_R$ and, as we are in Case (b), also taken so large that $\Sigma\cap \mathcal{V}_R$ is unbounded (in particular nonempty).

The proof of Theorem \ref{bi-halfspace} might not work here, because it could be that the function $f$ approaches its supremum only by attaining it on the boundary $\partial \Sigma$. Therefore the idea is to modify $f$ in a suitable way, so that the supremum of the new function is guaranteed to not be attained on $\partial \Sigma$ and also in such a way that the argument in the proof of the ``Bi-Halfspace'' Theorem~\ref{bi-halfspace} still goes through. The resulting argument, using the noncompactness to our advantage, is what we call an ``adiabatic trick'' since it involves tuning a certain length scale as slowly as needed together with estimates for the PDE.

To begin, recall that in the present case, $\Sigma\cap \mathcal{V}_R$ is now assumed to be an unbounded subset of $\Rmany$, so the extrinsic distance to $0\in\Rmany$ is an unbounded function on $\Sigma\cap \mathcal{V}_R$:
\beq\label{p-unbounded}
\sup_{p\in\Sigma\cap \mathcal{V}_R}\|p\|\subRmany=\infty.
\eeq

Since $\partial \Sigma$ is compact, there exists a radius $\rho >0$ large enough so that $\partial \Sigma \subseteq B_\rho(0) = \{x \in \R^{n+1} \colon \|x\|\subRmany \le \rho \}$. For every length scale $\ell > \rho > 0$ (which we soon plan to take as large as needed), let us define the $C^\infty(\Rmany)$ function $\chi_{\ell} \colon \R^{n+1} \to \R$ by
\beq
\chi_\ell (x) = \psi(\|x\|/\ell),
\eeq
where $\psi:[0,\infty)\to\R$ is a standard $C^\infty$ monotone increasing cut-off function $0\leq\psi\leq 1$ such that $\psi|_{[0,1]} \equiv 0$ while $\psi|_{[2,\infty)} \equiv 1$. Thus since  $\ell >\rho>0$ we have that $\chi_\ell$ vanishes inside the ball $B_\rho(0)$ and therefore also on $\partial \Sigma$. Furthermore, all ambient derivatives of $\chi_\ell$ are uniformly bounded with upper bounds depending only on $\ell$ (and of course $\psi$, which we fix once and for all):
\beq\label{bound_chi_ell}
\sup_{x\in\Rmany}\left\|\nabla^{\R^{n+1}}\chi_\ell(x)\right\|_{\Rmany} \leq \frac{C}{\ell}\quad  \text{ and } \quad \sup_{x\in\Rmany}\left|\Delta_{\R^{n+1}}\chi_\ell(x)\right| \leq\frac{C}{\ell^2} .
\eeq

For every $\ell > 0$, let us define the new function $f_\ell \colon \Sigma^n \to \R$ as follows. With $f$ as in Equation \eqref{definition_f} let $M \coloneqq \sup_{\Sigma} f$ and define:
\begin{equation}
f_\ell \lef p \rig \coloneqq f(p) + M \chi_\ell \lef p\rig,\quad p\in\Sigma.
\end{equation}
Note that the continuity and smoothness of $f_\ell$ are no worse than of $f$. Recall from (\ref{f-bounds}) that $f\leq \fracsm{R}{\xi}$ so that $f_\ell$ is also bounded:
\beq
\sup_\Sigma f_\ell \leq \fracsm{R}{\xi} + M <\infty.
\eeq

Also, since $f > R$ on $\Sigma\cap \mathcal{V}_R$ we have by (\ref{p-unbounded}) and by the fact that $\chi_\ell|_{\Rmany\setminus B_{2l}(0)}=1$:
\beq\label{avoid_M}
\forall \ell>\rho:\:\max_{\partial \Sigma} f_\ell \leq M <R + M<\sup_{\Sigma} f_\ell=\sup_{\Sigma\cap \mathcal{V}_R} f_\ell,
\eeq
using for the first equality that $\chi_\ell|_{\partial\Sigma} =0$ and for the last that $\sup_{\Sigma\setminus \mathcal{V}_R} f_\ell \leq R+M$. Thus we can now for each $\ell>\rho$ apply the Omori-Yau argument as in the proof of the ``Bi-Halfspace'' Theorem \ref{bi-halfspace} to the function $f_{\ell}$, this time in the boundary version, now that we by (\ref{avoid_M}) have verified the condition in Lemma \ref{mementomori}(i).

Suppose now that there exists $\ell_0 > 0$ such that there is at least one Omori-Yau sequence $p_k\in\Sigma\cap \mathcal{V}_R$ for $f_{\ell_0}:\Sigma\to\Reals$  with the property that $\|p_k\|\subRmany\to \infty$. Since $\chi_\ell$ is constant outside a compact subset of $\Rmany$, we see $\Delta_\Sigma f(p_k) = \Delta_\Sigma f_\ell(p_k)$ for all sufficiently large values of $k$, so that the argument in (\ref{laplace_contradict}) from the case without boundary applies.

Assume now conversely that for every $\ell > 0$, none of the Omori-Yau sequences have unbounded Euclidean norm. Then in consequence $f_{\ell}$ attains its maximum at some point $q_\ell\in\Sigma\cap \mathcal{V}_R\setminus \partial \Sigma$ so that $f_\ell(q_\ell) = \sup_{\Sigma\cap \mathcal{V}_R} f_\ell$. Note that then in fact $\|q_\ell\|\geq \ell$ must be the case, as follows from Equation (\ref{avoid_M}). Namely, inside $B_\ell(0)$ holds that $\chi_\ell = 0$, so we get $\sup_{B_\ell(0)} f_\ell \leq M < \sup_{\Sigma} f_\ell$ and thus the maximum must be attained outside of $B_\ell(0)$.

Now we do analysis on the sequence of maximum points $\{q_\ell\}$. By criticality we have $\nabla^{\Sigma} f_\ell (q_\ell) = 0$, so by \eqref{bound_chi_ell} and $\nabla^\Sigma \chi_\ell = \frac{1}{\ell}\psi'(\|p\|/\ell)\nabla^\Sigma \|p\|$:
\beq\label{surf_grad}
 \left\|\nabla^{\Sigma} f (q_\ell)  \right\|  = \left\|\nabla^{\Sigma} f_\ell (q_\ell) - M\nabla^{\Sigma} \chi_\ell (q_\ell)\right\| = M\left\| \nabla^{\Sigma} \chi_\ell (q_\ell)\right\| \leq \frac{CM}{\ell},
\eeq
where we also used
\beq\label{p-grad}
\left\|\nabla^\Sigma \|p\|\right\|=\Big\|\big( \nabla^{\R^{n+1}} \|p\| \big)^{\top}\Big\| \leq\big\|\nabla^{\R^{n+1}} \|p\|\big\|=1.
\eeq

As for estimating the Laplacian, we can compute:
\begin{align*}
\Delta_{\Sigma}\|p\| &= \di_{\Sigma} \lef \nabla^\Sigma \|p\| \rig \\
&= \di_{\Sigma} \lef \lef \nabla^{\R^{n+1}} \|p\| \rig^{\top} \rig \\
&=  \di_{\Sigma} \lef  \nabla^{\R^{n+1}}\|p\| -  \lef  \nabla^{\R^{n+1}}\|p\| \rig^{\perp} \rig \\
&= \frac{n}{\|p\|} + H \left\langle  \nabla^{\R^{n+1}}\|p\|, \nu \right\rangle.
\end{align*}

Therefore, since $\Sigma$ is a self-translater and hence $|H|\leq 1$, we get by Cauchy-Schwarz:
\begin{equation}\label{laplacian_distance_on_sigma}
\left|\Delta_{\Sigma}\|p\|\right| \le  \frac{n}{\|p\|} + 1, \quad p\in\Sigma.
\end{equation}

We thus get, using (\ref{p-grad}) and (\ref{laplacian_distance_on_sigma}) with $\|q_\ell\|\geq \ell$ :
\beq
|\Delta_\Sigma \chi_\ell(q_\ell)| \leq \left[\frac{\psi'(\|p\|/\ell)}{\ell}|\Delta_\Sigma \|p\|| + \frac{|\psi''|(\|p\|/\ell)}{\ell^2}\|\nabla^\Sigma \|p\|\|^2\right]_{\mid q_\ell}
\leq \frac{C'}{\ell}.
\eeq

Thus, since $\Delta_\Sigma f_\ell (q_\ell)\leq 0$ we get:
\beq\label{lim_laplace}
\lim_{\ell\to\infty} \Delta_\Sigma f(q_\ell) = \lim_{\ell\to\infty} \Delta_\Sigma f_\ell(q_\ell)-\lim_{\ell\to\infty} \Delta_\Sigma \chi_\ell(q_\ell)\leq 0.
\eeq

Therefore, by (\ref{surf_grad}) and (\ref{lim_laplace}), we can plug the sequence of maximum points $\{q_\ell\}$ directly into the same identity (\ref{main_identity}) derived in the course of the proof of the ``Bi-Halfspace'' Theorem \ref{bi-halfspace} for the $\partial\Sigma=\emptyset$ case, in order to get a contradiction. 

Since, both in Case (1) and in Case (2), we have thus reached a contradiction, we conclude that the  hypersurface $(\Sigma,\partial \Sigma)$ must in fact be compact.
\end{proof}


The following corollary completes the picture given by the ``Bi-Halfspace'' Theorem \ref{bi-halfspace}, providing a complete characterization  of all the possible couples of hyperspaces such that their intersection contains a properly immersed self-translater. In particular it shows that the ``Bi-Halfspace'' Theorem \ref{bi-halfspace} does not not hold anymore if we drop the assumption about the verticality of the halfspaces. 

\begin{corollary}\label{bi-halfspace_corollary}
Let $w_1, w_2 \in \mathbb{S}^n$ 
and let $H_1 \coloneqq H_{(0, w_1)}$ and $H_2 \coloneqq H_{(0, w_2)}$. 

Then there exists a properly immersed   self-translater without boundary contained in $H_1 \cap H_2$  if and only if one of the following conditions hold.
\begin{enumerate} 
\item $\langle w_1, e_{n+1}\rangle  > 0$  and $\langle w_2, e_{n+1} \rangle > 0$;
\item $\langle w_1, e_{n+1}\rangle  > 0$  and $\langle w_2, e_{n+1} \rangle = 0$;
\item $\langle w_1, e_{n+1}\rangle  = 0$  and $\langle w_2, e_{n+1} \rangle > 0$;
\item $\langle w_1, e_{n+1}\rangle  =  \langle w_2, e_{n+1} \rangle =0 $ and $w_1 \parallel w_2$.
\end{enumerate}
\end{corollary}

\begin{proof}
Let us first assume that none of the conditions $(1)$, $(2)$, $(3)$ and $(4)$ are satisfied. This means that  $\langle w_1, e_{n+1}\rangle =  \langle w_2, e_{n+1} \rangle = 0$ and $w_1 \nparallel w_2$ or one of the two scalar products is strictly negative. In the first case, we know from the ``Bi-Halfspace'' Theorem \ref{bi-halfspace} that there cannot be  properly immersed  self-translaters contained in $H_1 \cap H_2$. 

Let us assume that one of the two scalar products is strictly negative, say $\langle w_1, e_{n+1} \rangle < 0$. 
Then we claim that $H_1$ cannot contain any  properly immersed self-translater. This, in particular implies that $H_1 \cap H_2$ does not contained any  properly immersed self-translater. Indeed, by contradiction, assume that there exists a  properly immersed self-translater $\Sigma^n \subseteq H_1$. 
Then one can easily find a contradiction by using Lemma~\ref{comparison_principle} and comparing the time evolution of $\Sigma^n$ with the evolution of some suitably large sphere lying in $\R^{n+1} \setminus H_1$.

Let us now check that if any of $(1)$, $(2)$, $(3)$ or $(4)$ hold, then there exists a  properly immersed self-translater contained in $H_1 \cap H_2$. 

If $(1)$ holds, then consider for instance the bowl self-translater $U$. Since $U$ is asymptotic to a paraboloid at infinity, it is clear that, up to a translation in the $e_{n+1}$ direction, $U \subseteq H_1 \cap H_2$.

Let us now assume that $(2)$ or $(3)$ hold.  Without loss of generality, we can assume $H_1 = \{ x_1 \ge 0\}$ and $\langle w_2, e_{n+1}\rangle >0$.  Since we are assuming $\langle w_2, e_{n+1}\rangle >0$, we have that $P_2 \coloneqq \partial H_2$ is the graph of an affine function $f$ defined over  $\{x_{n+1} = 0\}$. More precisely, let $w_2 = (w_{2, 1}, \dots, w_{2, n}, w_{2, n+1})$. Then $f$ is defined as
$$
f(x_1, \dots, x_n) \coloneqq - \frac{x_1 w_{2,1} + x_2  w_{2, n}\dots + x_n w_{2,n}}{w_{2,n+1}}.
$$
For any $L >0$, let us define the slab $S_L \coloneqq (0, L) \times \R^{n-1}$. Note that on $S_L$ the function $f|_{S_L}$ is bounded from above by the function 
$$
g_L(x_1, \dots x_n) \coloneqq L\frac{|w_{2,1}|}{w_{2,n+1}} - \frac{x_2 w_{2,2}\dots + x_n w_{2,n}}{w_{2,n+1}}
$$
 and clearly $\nabla g_L = \frac{1}{w_{2,{n+1}}} (0, w_{2,2}, \dots,  w_{2,n})$. Note that $\nabla g_L$ does not depend on $L$. 
Now take $L$ large enough so that there exists a tilted grim reaper cylinder $\Sigma$ which is the graph of a function defined on $S_L$ and such that it grows linearly in the direction of $\nabla g_L$ and with the same slope of $g_L$ (for a detailed description of tilted grim reaper cylinders, see \cite{gama_martin} and \cite{blt18}). Then, since $\Sigma$ is the graph of a function which is strictly convex w.r.t. the first variable $x_1$, it can be chosen in such a way that it lies above the graph of $g_L$ and, in particular, inside $H_2$. Moreover, by construction, $\Sigma$ is also contained in $H_1$. 

If $(4)$ holds, then observe that $P \coloneqq \partial H_1 = \partial H_2 $ is a translater contained in $H_1 \cap H_2$.
\end{proof}

\section{On the Convex Hulls of Self-Translaters}\label{sec:basics}

In this section we want to study the convex hulls of self-translaters. We will derive a sort of ``convex hull property'' for compact  self-translaters and then we will discuss the classification of the convex hulls of (possibly noncompact) self-translaters with compact boundary, proving Theorem \ref{convex_hull_noncomp_trans}. Those two results have been inspired by the theory of classical minimal submanifolds of the Euclidean space.  They both show that, up to projecting onto the hyperplane $\R^n \times \{0\}$,  the convex hull of a self-translater behaves quite similarly to the convex hull of a minimal submanifold of $\R^{n+1}$.

\subsection{Convex Hulls of Compact Self-Translaters}
The first lemma is a well-known fact about self-translaters and can be proved in several different ways, but, at least to our knowledge, they are all based on some version of the maximum principle. For the sake of completeness we include a proof, close in spirit to an argument given in \cite{pyo}.

\begin{lemma}\label{lemma_boundary}
 Let $(\Sigma^n,\partial\Sigma) $ be a compact $e_{n+1}$-directed self-translater in $\Rmany$.

Then $\partial \Sigma \ne \emptyset$ and 
$$
\max_{\overline{\Sigma}}  x_{n+1} = \max_{\partial\Sigma}x_{n+1}.
$$ 
\end{lemma}

\begin{proof}
Recall that given a function $f \in C^1(\R^{n+1})$, the gradient $\nabla^{\Sigma} f|_{\Sigma}$ is given by 
\begin{equation}\label{equation_rest_gradient}
\nabla^{\Sigma} f|_{\Sigma} = \lef \nabla f \rig^{\top},
\end{equation}
where $\lef \nabla f \rig^{\top}$ is the projection of $\nabla f $ on the tangent bundle of $\Sigma$.

If we  apply \eqref{equation_rest_gradient} to the coordinate function $x_{n+1}$, we get
\begin{equation}
\nabla^{\Sigma} x_{n+1} = \sbase_{n+1}^\top.
\end{equation}
Let $\oframe_1, \dots, \oframe_n$ be a orthonormal frame on $\Sigma$ and let $\un$ be a unit normal vector field. 

Then, using \eqref{translater_equation}, we have
\begin{align*}
\Delta_{\Sigma} x_{n+1}
&= \di_{\Sigma}(\sbase^{\top}_{n+1}) =  \di_{\Sigma}(\sbase_{n+1} - \sbase_{n+1}^\perp) \\
&= -\Sigma_{j=1}^n \langle \nabla_{\oframe_j}\langle \sbase_{n+1}, \un \rangle \un, \oframe_j \rangle \\
&= - \langle \sbase_{n+1}, \un\rangle \Sigma_{j=1}^n \langle \nabla_{\oframe_j} \un, \oframe_j \rangle \\
&= H^2.
\end{align*}

Therefore $x_{n+1}$ is a subharmonic function on $\Sigma$, and hence by the strong maximum principle it cannot have any interior maximum points.
\end{proof}

Now let us show a new ``convex hull'' property for self-translaters, in the same spirit as the classical one for minimal hypersurfaces. Let us first remind the reader of the minimal hypersurface case.

\begin{proposition}(See e.g. Proposition 1.9 in \cite{cm-min}). \label{convex_hull_cm}
If $\Sigma^n\subseteq \R^{n+1}$ is a compact minimal hypersurface with boundary, then $\Sigma\subseteq\conv(\partial\Sigma)$, where $\conv(\partial \Sigma)$ is the convex hull of $\partial \Sigma \subseteq \R^{n+1}$.
\end{proposition}

Read verbatim, such a statement is ostensibly wrong for self-translaters, as e.g. seen by taking the (compact) pieces of the Altschuler-Wu bowl solution below planes perpendicular to $e_{n+1}$. Nonetheless, we do have the following modified version. We will by $\pi \colon \R^{n+1} \to \R^n$ denote the standard orthogonal projection $\pi(x_1, \dots, x_n, x_{n+1}) \coloneqq (x_1, \dots, x_n)$.


\begin{proposition}\label{prop_convex_hull}
Let $\Sigma^n \subseteq \R^{n+1}$ be a compact $e_{n+1}$-directed self-translater with boundary $\partial \Sigma\neq \emptyset$.

Then
$$
\Sigma \subseteq \conv\lef \pi\lef \partial \Sigma \rig\rig \times (-\infty, \max_{\partial \Sigma} x_{n+1}],
$$ 
where $\conv\lef\pi\lef \partial \Sigma \rig\rig$ is the convex hull of $\pi(\partial \Sigma) \subseteq \R^n$.
\end{proposition}

\begin{proof}
Let $\tilde{\R}^{n+1} \coloneqq \ilmanen = \lef \R^{n+1}, \tilde{h} \rig$ be the so-called Huisken-Ilmanen space. It  plays an important role due to the following well-known correspondence: $\Sigma^n \subseteq \R^{n+1}$ is a unit speed self-translating surface in the $x_{n+1}$-direction if and only if $\Sigma$ is a minimal submanifold of $\tilde{\R}^{n+1}$. See for instance \cite{sha} for a proof in the case $n = 2$ or \cite{jpg} for the general case.

Observe that given a function $f \in C^1 \lef \R^{n+1} \rig$, the gradient $\tilde{\nabla} f $ of  $f$  w.r.t. the metric  $\tilde{h}$ is given by
\begin{equation}\label{formula_conformal_gradient}
\tilde{\nabla} f = e^{- \frac{2}{n} x_{n+1} } \nabla f.
\end{equation}

We can now compute  $\Delta_{\tilde{\Sigma}}x_j$, for $j = 1, \dots, n$, using \eqref{formula_conformal_gradient} and \eqref{equation_rest_gradient}.
\begin{align*}
\Delta_{\tilde{\Sigma}} x_{j}
&= \di_{\tilde{\Sigma}} \lef  \nabla^{\tilde{\Sigma}} x_j \rig \\
&= \di_{\tilde{\Sigma}} \lef \lef \tilde{\nabla} x_j \rig^T \rig\\
&= \di_{\tilde{\Sigma}} \lef e^{-\frac{2}{n}x_{n+1}} \sbase_{j}^{\top} \rig \\ 
&= -\frac{2}{n} e^{-\frac{2}{n}x_{n+1}} \tilde{h} \lef \nabla^{\tilde{\Sigma}}x_{n+1}, \sbase^{\top}_j \rig  + e^{-\frac{2}{n}x_{n+1}} \di_{\tilde{\Sigma}} \lef \sbase^{\top}_j \rig \\
&= - \frac{2}{n} \tilde{h} \lef \nabla^{\tilde{\Sigma}} x_{n+1}, \nabla^{\tilde{\Sigma}} x_j \rig + e^{-\frac{2}{n} x_{n+1}} \di_{\tilde{\Sigma}}\lef \sbase_j \rig.
\end{align*}

Note that $ \di_{\tilde{\Sigma}}\lef \sbase_j^\top \rig =  \di_{\tilde{\Sigma}}\lef \sbase_j \rig$ because $\tilde{\Sigma}$ is minimal in $\tilde{\R}^{n+1}$. Moreover note that $\di_{\tilde{\Sigma}}\lef \sbase_j \rig = 0$ since $\sbase_j$ is a Killing field on $\tilde{\R}^{n+1}$, for every $j = 1, \dots, n$. Indeed let $\mathcal{L}$ denote the Lie derivative. Then we have
\begin{equation}\label{eq_lie_derivative}
\mathcal{L}_{e_j} \tilde{h} = \mathcal{L}_{e_j} \lef  e^{\frac{2}{n}x_{n+1}} h\rig  =   e^{\frac{2}{n}x_{n+1}} \mathcal{L}_{e_j} h = 0.
\end{equation}

Therefore for  each $j = 1, \dots, n$, the coordinate function $x_j$ satisfies the following linear elliptic PDE:
$$
\Delta_{\tilde{\Sigma}} x_{j} + \frac{2}{n} \tilde{h} \lef \nabla^{\tilde{\Sigma}} x_{n+1}, \nabla^{\tilde{\Sigma}} x_j \rig = 0,\quad j=1,\ldots, n.
$$
From the maximum principle we have that each $x_j$, for $j=1,\ldots,n$, attains its maximum and minimum on $\partial \Sigma$. This, together with Lemma \ref{lemma_boundary}, concludes the proof.
\end{proof}

\begin{remark}
Observe that for the proof of Proposition  \ref{prop_convex_hull} one could alternatively have proven by contradiction that $x_j$, for $j=1,\ldots,n$ has no interior maxima and minima using the Lemma \ref{tangency_principle} and comparing with vertical translating planes.  This is not surprising, since the Principle of Separating Tangency is another manifestation of the strong maximum principle for quasilinear elliptic equations.

Note also that only $x_i$ when $i=1,\ldots, n$ works, and that one could not use $x_{n+1}$ in Proposition \ref{prop_convex_hull}, as the similar computation as in \eqref{eq_lie_derivative} performed for $e_{n+1}$ shows that $e_{n+1}$ is not a Killing field of $\tilde{\Reals}^{n+1}$.
\end{remark}

The ``convex hull'' property provides immediately the following monotonicity of topology for compact self-translaters. 

\begin{corollary}
Let $\Sigma^n \subseteq\R^{n+1}$ be a compact self-translater. Let $C \subseteq\R^n$ be a compact convex set such that $C \cap \pi\lef \partial \Sigma \rig  = \emptyset$, where $\pi$ is the usual projection $\pi \colon (x_1, \dots, x_n, x_{n+1}) \to (x_1, \dots, x_n)$. 

Then the inclusion map $ i \colon \lef C \times \R\rig \cap \Sigma \hookrightarrow \Sigma$ induces an injection on the $(n-1)$-st homology group.
\end{corollary}

\begin{proof}
The proof is very similar to the one of Lemma 1.11 in \cite{cm-min}. 
\end{proof}


\subsection{Convex Hulls of Noncompact Self-Translaters}

Note that the results in the preceding section were all about compact self-translaters. We will now study the convex hull property in the noncompact case (Theorem~\ref{convex_hull_noncomp_trans}).   Also, as mentioned in the introduction, this result was inspired by the classical result for minimal submanifolds in Euclidean space proved by Hoffman and Meeks in \cite{hoffman-meeks} that we recall here.

\begin{theorem}[Hoffman-Meeks: Theorem 3 in \cite{hoffman-meeks}]\label{hoffman_meeks_theorem}
Let $\Sigma^n \subseteq \R^{n+1}$ be a properly immersed connected minimal submanifold whose (possibly empty) boundary $\partial \Sigma$ is compact.
Then exactly one of the following holds:
\begin{enumerate}
\item $\conv(\Sigma) = \R^{n+1}$,
\item $\conv(\Sigma)$ is a halfspace,
\item $\conv(\Sigma)$ is a closed slab between two parallel hyperplanes,
\item $\conv(\Sigma)$ is a hyperplane,
\item $\conv(\Sigma)$ is a compact convex set. This case occurs precisely when $\Sigma$ is compact.
\end{enumerate}
Moreover, when $n = 2$, $\partial \Sigma$ has nonempty intersection with each boundary component of $\conv(\Sigma)$.
\end{theorem}

Recall again that from the known examples (see Section \ref{sec:prelims}), we cannot hope to have the same characterization of the convex hulls of self-translaters. But we can characterize the convex hull of the projection onto the hyperplane $\R^n \times \{0\}$. This is the content of Theorem \ref{convex_hull_noncomp_trans} and the proof is based on the ``Bi-Halfspace'' Theorem \ref{bi-halfspace}.

\begin{remark}
Note that the last statement of Theorem \ref{hoffman_meeks_theorem}, which follows from the Halfspace Theorem (Theorem 1 in \cite{hoffman-meeks}), does not have a straightforward equivalent in the context of self-translaters. Indeed it is natural to ask if it is true or not that given a connected, properly immersed, $2$-dimensional self-translater $\Sigma^2 \subseteq \R^3$ with compact boundary, $\pi \lef \partial \Sigma\rig$ has nonempty intersection with each topological boundary component of $\conv \lef  \pi\lef \Sigma \rig\rig$. The answer is negative. Indeed one can easily build a counterexample by taking as $\Sigma$ a grim reaper cylinder with a compact set removed. 
\end{remark}


Before giving the proof of Theorem \ref{convex_hull_noncomp_trans}, let us first prove the following simple characterizations of compact self-translaters.

\begin{lemma}[Characterization of Compact Self-Translaters]\label{lemma_ceiling}
Let $(\Sigma^n,\partial\Sigma)$ be a properly immersed, connected self-translater with compact boundary.  Then the following are equivalent. 

\begin{enumerate}
\item $\Sigma$ is compact.
\item $\sup_{\Sigma} x_{n+1} <  \infty$.
\item $\Sigma$  is contained in a cylinder of the kind $K \times \R$, where $K \subseteq \R^n$ is a  compact set.
\end{enumerate}
\end{lemma}

\begin{proof}[Proof of Lemma \ref{lemma_ceiling}]
$(1) \Rightarrow (2)$. If $\Sigma$ is compact, then clearly  $\sup_{\Sigma} x_{n+1} <~\infty$. 

$(2) \Rightarrow (3)$.  Let us assume that  $\sup_{\Sigma} x_{n+1} <  \infty$.   Let $R > 0$ be a radius large enough such that $\pi \lef \partial \Sigma \rig \subseteq B_R(0)$, where $B_R(0)$ is the ball of radius $R>0$ in $\R^{n} \times \{0\}$, centered in $0$.

Let us consider the winglike self-translaters $W_R$ from \cite{CSS}, which we translate so that $\inf_{p\in W_R} x_{n+1}(p)=0$. Let us define  the one-parameter family of wing-like self-translater $\{W_{R, s} \}_{s \in \R}$, where $W_{R, s} \coloneqq W_R + s\, e_{n+1}$. Clearly we have that 
\begin{equation}\label{intersection_wing_sigma}
W_{R, s} \cap \Sigma = \emptyset,
\end{equation}
for every $s >\sup_\Sigma x_{n+1}$.   
 Assume by contradiction that  there exists $s \in \R$ such that $W_{R, s} \cap \Sigma \ne \emptyset$. 
Since $\Sigma$ is properly immersed, there exists 
$$
s_0 \coloneqq \max \{ s \in \R \colon W_{R, s} \cap \Sigma \ne \emptyset \}.
$$
This leads to a contradiction, thanks to Lemma \ref{tangency_principle}. 
Therefore  \eqref{intersection_wing_sigma} holds for every $s \in \R$ and thus $\Sigma$ is contained in the cylinder $B_R(0) \times \R$.

$(3) \Rightarrow (1)$ Let us assume that $\Sigma \subseteq K \times \R$, for some compact set $K \subseteq \R^{n}$.  
Let us assume by contradiction that $\Sigma$ is not compact. This implies that $\sup_{\Sigma}x_{n+1} = \infty$ or $\inf_{\Sigma}x_{n+1} = -\infty$. Let us consider the first case (the other case is similar).  

Since $\partial \Sigma$ is compact, we can assume w.l.o.g.  that $\partial \Sigma \subseteq \{x_{n+1} \le -1\}$. For every $R >0$, let $W_{R, 0}$ be the winglike self-translater with neck size $R>0$ and such that $\min_{W_{R, 0}} x_{n+1} = 0$. Let us consider the family $\{W_{R, 0} \}_{R> 0}$. Note the difference with the winglike self-translaters family above: now the ``height'' is fixed and $R>0$ is a parameter. 

Observe that $W_{R, 0} \cap \lef K \times \R\rig = \emptyset $ for $R>0$ large enough. Therefore $W_{R, 0}  \cap \Sigma = \emptyset$, for $R>0$ large enough. On the other hand, since $\Sigma$ is connected and since  $\sup_{\Sigma}x_{n+1} = \infty$, there exists $r > 0$ small enough such that $W_{r, 0} \cap \Sigma \ne \emptyset $. Since $\Sigma$ is properly immersed, there exists
$$
r_{0} \coloneqq \max \{ r > 0 \colon W_{r, 0} \cap \Sigma \ne \emptyset \}.
$$
Note that since $\partial \Sigma \subseteq \{x_{n+1} \le -1\}$ every point in the intersection $W_{r_0, 0} \cap \Sigma$ is an interior point. This contradicts Lemma \ref{tangency_principle}.
\end{proof}


\begin{proof}[Proof of Theorem \ref{convex_hull_noncomp_trans}]\label{general_proof}
First of all, observe that the ``if and only if'' part in Theorem \ref{convex_hull_noncomp_trans}'s Case (5) follows directly from Lemma \ref{lemma_ceiling}.

Take $\Sigma^n\subseteq\Rmany$ possibly with compact boundary $\partial \Sigma$. The vertical projection of the convex hull of $\Sigma^n$, or equivalently convex hull of the vertical projection, can be written as the intersection of all vertical halfspaces in $\Rmany$ which contain it:
\beq\label{vertichull}
\conv (\pi( \Sigma))\quad =\quad\bigcap_{\left\{H:\: \Sigma\subseteq H \:\mathrm{vertical\:halfspace\:of\:}\Rmany\right\}} \hspace{-52pt}\pi(H)\quad\quad\quad\subseteq \Rfew.
\eeq
If the index set is empty we get $\conv (\pi( \Sigma)) =\Rfew$ and arrive at Case (1). So, we assume now that this is not the case.

We will now deduce that in the intersection (\ref{vertichull}) all the involved halfspaces $H\subseteq\Rmany$, and hence all the $\pi(H)\subseteq\Rfew$, are in fact (anti-)parallel halfspaces, unless we are in Case (5). Namely, let $H_1$ and $H_2$ be any two vertical closed halfspaces of $\R^{n+1}$, i.e. such that $P_1 \coloneqq \partial H_1$ and $P_2 \coloneqq \partial H_2$ are two hyperplanes both containing $e_{n+1}$, and with $\Sigma^n \subseteq H_1 \cap H_2$. Then if $H_1$ and $H_2$ were not (anti-)parallel, the compact boundary version of the ``Bi-Halfspace'' Theorem \ref{bi-halfspace_boundary} would imply that $\Sigma^n$ is compact (and note that necessarily $\partial \Sigma\neq\emptyset$ too), so that we would arrive at Case (5).

We may thus finally assume that we are not in Case (1) nor in Case (5). Since all vertical halfspaces in $\Rmany$ which contain $\Sigma^n$ are then mutually (anti-)parallel, so are all the $(n-1)$-dimensional hyperplanes $\pi(H)$ in $\Rfew$ and the intersection in (\ref{vertichull}) is now easy to evaluate: One of the Cases (2), (3) or (4) must occur. This concludes the proof of Theorem \ref{convex_hull_noncomp_trans}.
\end{proof}

\begin{remark}

Even though Theorem \ref{convex_hull_noncomp_trans} was inspired by Theorem \ref{hoffman_meeks_theorem}, our proof is quite different from the original proof of Hoffman and Meeks in \cite{hoffman-meeks}. 

First of all, observe  that  the ``if and only if'' of point $(5)$ in Theorem \ref{hoffman_meeks_theorem} is trivial, but one implication of the  ``if and only if'' of point $(5)$ in Theorem \ref{convex_hull_noncomp_trans} is not completely obvious.  

But the most important difference is that the proof of Hoffman and Meeks is an elaborate application of the maximum principle for the nonlinear minimal hypersurface equation, while our proof is based on the Omori-Yau maximum principle. 

In the Appendix \ref{appendix} we provide an alternative proof of Theorem \ref{convex_hull_noncomp_trans} in the case $n=2$ which is based on Lemma \ref{tangency_principle} and it is closer in spirit to the original proof of Hoffman and Meeks. We also explain why it is hard to extend it to higher dimension.
 \end{remark}


\section{Appendix}\label{appendix}

In this appendix we present an alternative  proof of Theorem \ref{convex_hull_noncomp_trans}, which works only in the case $n=2$. 

Before presenting the proof, let us recall the following simple property about winglike self-translaters.
\begin{lemma}\label{auxiliary_lemma}
Let $R > 0$ and let $W_R \subseteq \R^{n+1}$ be the wing-like self-translater as in \cite{CSS} and \cite{niels}. Let us denote by $R^*> R$ the radius at which the coordinate function $x_{n+1}$ attains the minimum on $W_R$. 

Then
$$
R^* - R \le \frac{\pi}{2}.
$$
\end{lemma}
\begin{proof}
The proof of this lemma is contained in the proof of Lemma 2.1 in \cite{niels}. 
\end{proof}


\begin{proof}[Proof of the $2$-dimensional version of Theorem \ref{convex_hull_noncomp_trans}]  Let $\Sigma^2 \subseteq \R^3$ be a properly immersed self-translater with compact boundary $\partial \Sigma$. 
In the theorem,  let us assume that the Cases $(1)$, $(4)$ and $(5)$ do not occur. We want to show that then Case $(2)$ or Case $(3)$ must occur. 
Let $H_1$ and $H_2$ be two closed halfspaces (here: halfplanes) in $\R^2$ such that $\conv (\pi( \Sigma)) \subseteq H_1 \cap H_2$. Let $P_1 \coloneqq \partial H_1$ and $P_2 \coloneqq \partial H_2$. In order to show that case $(2)$ or case $(3)$ must occur, it is sufficient to show that the lines $P_1$ and $P_2$ are parallel.

Let us assume by contradiction that $P_1$ and $P_2$ are not parallel. 
The idea is to show that  $\Sigma$ must be then contained in a halfspace of the kind $\{x_{3} \le K \}$ for $K$ large enough. This will contradict Lemma \ref{lemma_ceiling}.

Let us consider $\tilde{H}_1 \coloneqq \pi^{-1} \lef H_1 \rig = H_1 \times \R$ and $\tilde{H}_2 \coloneqq \pi^{-1} \lef H_2 \rig = H_2 \times \R$. Note that $\tilde{H}_1$ and $\tilde{H}_2$ are closed halfspaces of $\R^{3}$ and $\Sigma \subseteq \tilde{H}_1 \cap \tilde{H}_2$. Moreover we will denote  $\tilde{P}_1 \coloneqq \pi^{-1}\lef P_1\rig = P_1 \times \R$ and $\tilde{P}_1 \coloneqq \pi^{-1}\lef P_1\rig = P_1 \times \R$. Note that $\tilde{P}_1$ and $\tilde{P}_1$ are affine planes in $\R^{3}$, both parallel to the $x_{3}$-axis. Without loss of generality, we may assume that   $\tilde{P}_1 \cap \tilde{P}_2$ is the $x_{3}$-axis.

From Lemma \ref{tangency_principle}, since $\tilde{P}_1$  and  $\tilde{P}_2$ are both self-translaters, $\Sigma$ does not have any interior point in common with them, i.e.  $\lef \Sigma \setminus \partial \Sigma\rig  \cap \lef \tilde{P}_1 \cup \tilde{P}_2\rig = \emptyset$. 
%
%
For every $R>0$, let $S_R \subseteq H_1 \cap H_2 \subseteq \R^2$ be the unique circle of radius $R>0$ and tangent to $P_1$ and $P_2$ and let $p_R \in H_1 \cap H_2$ be the center of $S_R$. Moreover let $\bar{B}_R(p_R)$ be the closed ball of center $p_R$ and radius $R>0$. Observe that since $S_R$ is tangent to $P_1$ and $P_2$, $\lef H_1 \cap H_2\rig \setminus \bar{B}_R$ consists of two connected regions, one bounded and the other one unbounded.  Let us denote by $A_R$ the the closure of the bounded region. Observe that 
$$
\lim_{R \searrow 0} \diam A_R = 0.
$$

 For each $R>0$, let  $W_{R}$ be the wing-like self-translater such that it is rotationally symmetric around $\{p_R\} \times \R$ and $\min_{W_{R}} x_{3} = 0$ and $R>0$ is the aperture of the ``hole''. Moreover, let $R^*$ be the radius as in Lemma \ref{auxiliary_lemma}, i.e. $x_{3} = 0$ on the circle $S_{R^*}(p_R)$ of radius $R^*$ and centered in $p_R$.
$$
\tilde{W}_{R} \coloneqq W_{R} \cap \lef A_R \times \R \rig.
$$
It is easy to check that $\tilde{W}_{R} \subseteq \tilde{H}_1 \cap \tilde{H}_2 $ is compact and $\partial \tilde{W}_{R} \subseteq \tilde{P}_1 \cup \tilde{P}_2$.

Since $\partial \Sigma $ is compact, up to a translation in the $x_{3}$-direction, we can assume $\partial \Sigma \subseteq \{x_{3} \le - 1 \}$.

Moreover, since $\Sigma$ is properly immersed, we have that there exists $r>0$ small enough, such that
$$
\tilde{W}_{r} \cap \Sigma = \emptyset.
$$
Consider the $1$-parameter family $\{\tilde{W}_{R}\}_{R>0}$. Using Lemma \ref{tangency_principle} and a standard argument, we have that $\tilde{W}_{R} \cap \Sigma = \emptyset $ for every $R >0$. 

From Lemma \ref{auxiliary_lemma}, we have that $S_{R^*}(p_R) \cap A_R \ne \emptyset$,  for every $R>0$ such that $\dist(p_R, 0) > \frac{\pi}{2}$. Moreover the family of compact sets $\{S_{R^*}(p_R) \cap A_R \}_{R>0}$ swipes out the whole plane $\R^2 \times \{0\}$, i.e.
$$
\bigcup_{R>0} S_{R^*}(p_R) \cap A_R = \R^2 \times \{0\}.
$$
Therefore we have that 
\begin{equation}\label{1.8.18}
\Sigma \subseteq \{x_{3} \le 0\}.
\end{equation}

Recall that $\Sigma$ is not compact, because we are assuming that $(1), (4)$ and $(5)$ do not hold. This generates a contradiction because from \eqref{1.8.18} and from Lemma \ref{lemma_ceiling}, we have that $\Sigma$ must be compact.

Therefore we showed that if $(1), (4)$ and $(5)$ do not hold, then $(2)$ or $(3)$ must occur. 
\end{proof}

Observe that the above proof is quite similar to the proof in \cite{hoffman-meeks}, but it works only for $n=2$. Indeed note that it is not possible to naively generalize the above proof to higher dimension. The problem is that it is not possible to define the set $A_R$. Indeed let us assume that $n \ge 3$ and let $H_1$ and $H_2$ be halfspaces of $\R^n$ as in the proof above, and let $P_1$ and $P_2$ be their boundaries respectively. Then let $B$ a closed ball such that $S = \partial B$ is tangent both to $P_1$ and to $P_2$ and such that $B \subseteq H_1 \cap H_2$. Then $\lef H_1 \cap H_2 \rig \setminus B $ is connected. Therefore the argument of the proof above does not work.   

However, with a straightforward generalization of the argument above, one can prove a weaker version of Theorem \ref{bi-halfspace_boundary}. More precisely, one can prove the following result.

\begin{theorem} Let $(\Sigma^n, \partial \Sigma)$ be a properly immersed  connected self-translating n-dimensional hypersurface  in $\R^{n+1}$. Let $\mathcal{C} \subseteq \R^n$ be a half-cone, i.e. 
$$
\mathcal{C} = \{ x \in \R^n \colon \mathrm{angle}(x , w) < \alpha \}
$$
for some $w \in \mathbb{S}^{n-1}$ and some angle $\alpha \in (0, \frac{\pi}{2})$.

Then if $\Sigma^n \subseteq \mathcal{C} \times \R$ it must be compact.
\end{theorem}

\begin{remark}
The proof of Hoffman and Meeks works in any dimension because they used as barriers  solutions of a Dirichlet problem for the minimal hypersurface equation. 

Indeed it is known that for every  bounded, convex, $C^2$ domain $\Omega \subseteq \R^n$, and for every $\varphi \in C^0\lef \partial \Omega  \rig$ there exist a solution $u \in C^2\lef \Omega \rig \cap C^0\lef \bar{\Omega} \rig$ of the following Dirichlet problem. 
\begin{equation}
\begin{cases}
\di \lef \frac{D u }{\sqrt{1 + |Du|^2}} \rig = 0 \qquad &\text{in } \Omega  \\
u|_{\partial \Omega} = \varphi  \qquad &\text{on } \partial \Omega.
\end{cases}
\end{equation}
For more details, see Section 16.3 in  \cite{gt}.

In our case we would have needed to solve a Dirichlet problem of the kind \eqref{side_trans_dirichlet}. Indeed it is easy to verify that a self-translater which is graphical w.r.t. a direction orthogonal to the moving direction $e_{n+1}$ is the graph of a function satisfying the PDE below in \eqref{side_trans_dirichlet}. Unfortunately in this case there is no general existence result, even assuming the initial data to be smooth. See Proposition \ref{prop_counter_example_dirichlet} below. Therefore we firstly resorted to building barriers carefully from the known family of wing-like self-translaters, the drawback being that this procedure only works in the case $n=2$, as we already explained. This motivated us to look for a different approach and led us to the proof of the ``Bi-Halfspace'' Theorems \ref{bi-halfspace}--\ref{bi-halfspace_boundary} and consequently to the proof of Theorem \ref{convex_hull_noncomp_trans}, as presented in the main parts (see Section \ref{general_proof}) of this paper.
\end{remark}
\begin{proposition}\label{prop_counter_example_dirichlet}
There exists $\Omega \subseteq \R^n$ bounded, convex with smooth boundary $\partial \Omega$ and there exists $\varphi \in C^{\infty}\lef \partial \Omega \rig$ such that there exists no function $u \in C^2\lef \Omega \rig \cap C\lef \bar{\Omega} \rig$, $u = u(y_1, \dots, y_n)$, satisfying the following Dirichlet problem.
\begin{equation}\label{side_trans_dirichlet}
\begin{cases}
\di\lef \frac{Du}{\sqrt{1 + |Du|^2}} \rig = \frac{u_{y_1}}{\sqrt{1 + |Du|^2}} \qquad &\text{ in } \Omega \\
u|_{\partial \Omega} = \varphi \qquad & \text{ on } \partial \Omega
\end{cases}
\end{equation}
\end{proposition}

%
\begin{figure}
\centering
\begin{tikzpicture}[ scale = 0.45]
      \draw[thick, scale=1,domain=-3:3,smooth,variable=\y,]  plot ({\y*\y},{\y});
       \draw[ scale=1,domain=-3:3,smooth,variable=\y,]  plot ({\y*\y +8},{\y});
             \draw[ yscale=2] (17,0) circle (1.5);
       \draw[ scale=1,domain=-3:3,smooth,variable=\y,]  plot ({\y*\y - 8},{\y});
      \draw[ yscale=2] (1,0) circle (1.5);
      \draw[thick, yscale=2] (9,0) circle (1.5);
      \draw[thick, rotate around={45:(5,0)},red] (5,-0.63) ellipse (20pt and 89pt);
     \draw [thick] (3.17, 1.65)  -- (3.17, -5);
      \draw [thick ] (7.72, -2.5)  -- (7.72, -5);
       \draw[thick, red, xscale=2, yscale=0.5, fill=gray!20] (2.73,-10) circle (1.13);
     \draw[thick] (-2, -6) -- (10, -6) -- ( 12, -4) -- (0,-4) -- (-2,-6);
	\draw[thick, ->] (14, -5) -- (16, -5) ;
\draw [] ( 2, -5 ) node [anchor = center]{$\Omega$};  
\draw [] ( 5, 3) node [anchor = center]{$U_0$}; 
\draw [] (13, 3) node [anchor = center]{$U_t$}; 
\draw [] (-3, 3) node [anchor = center]{$U_t$}; 
\draw [] ( 15.5, -5 ) node [anchor = north]{$e_{n+1}$};  
\draw [] ( -3, -5 ) node [anchor = north]{$Q$};  
\draw [] ( 6, 0) node [anchor = south]{$\Gamma$};  
\end{tikzpicture}
\caption{}
\label{counter_example_dirichlet}
\end{figure}

\begin{proof}
Let $U \subseteq \R^{n+1}$ be the bowl self-translater. Let $P$ be an affine hyperplane of $\R^{n+1}$ such that it is not parallel to $e_{n+1}$ but not orthogonal to $e_{n+1}$. Let $Q$ be another hyperplane parallel to $e_{n+1}$ and such that $P$ is graphical over $Q$. 

Let $\Gamma \coloneqq U \cap P$. Observe that, up to translating $P$ in the direction of $e_{n+1}$, we can assume $\Gamma \ne \emptyset$. Moreover, we can take $P$ such  that $\Gamma = \partial U_\Gamma$, where $U_\Gamma \subseteq U$ is a bounded subset of $U$ which is not graphical over $Q$. 

Let $\pi_Q \colon \R^{n+1} \to Q$ be the orthogonal projection onto $Q$. 

Since $U$ is a convex hypersurface, we have that $\pi \lef \Gamma \rig $ is the boundary of some bounded convex domain $\Omega \subseteq Q$ (see Figure \ref{counter_example_dirichlet}). Since $P$ is graphical over $Q$, we have that $\Gamma$ is the graph of some function $\phi \colon \partial \Omega \to \R$.  

Let $y_1, \dots, y_n$ be Cartesian coordinates on $Q$ such that the coordinate $y_1$ coincides with $ x_{n+1}$.  

Now assume by contradiction that there exists a solution $u$ for  the Dirichlet problem  \eqref{side_trans_dirichlet}.

Therefore $\graph\lef u \rig $ is a compact self-translater with unit velocity $e_{n+1}$ with boundary $\Gamma$. 

Now for every $t \in \R$ define $U_t \coloneqq U + t e_{n+1}$. Observe that the family $\{U_t\}_{t\in \R}$ foliates $\R^{n+1}$.

Since $\graph \lef u \rig$ is compact and each $U_t$ is properly immersed, there exist
$$
t_{\min} \coloneqq \min \{ t \in \R \colon U_t \cap \graph \lef u \rig \ne \emptyset \}
$$
and
$$
t_{\max} \coloneqq \max \{ t \in \R \colon U_t \cap \graph \lef u \rig \ne \emptyset \}.
$$

If $t_{\min} < 0$, then every point $p \in  U_{t_{\min}} \cap \graph \lef u \rig$ would be an interior point of $\graph\lef u\rig$. From Lemma \ref{tangency_principle}, we would have that $\graph\lef u\rig \subseteq U_{t_{\min}}$, and therefore $\Gamma =\partial \lef \graph\lef u\rig\rig \subseteq U_{t_{\min}}$. But this is a contradiction because $\Gamma \subseteq U_0 = U$. Therefore $t_{\min} = 0$. 

With a similar argument one can show that  $t_{\max} = 0$. Therefore $\graph\lef u\rig = U_\Gamma \subseteq U_0$. But this is a contradiction, because $U_\Gamma$ is not graphical by construction.

\end{proof}

\end{document}